\documentclass[12pt]{amsart} 
\usepackage{amscd}
\usepackage{amssymb}
\usepackage{a4wide}
\usepackage{amstext}
\usepackage{amsthm}
\usepackage{xcolor}
\usepackage{cite}
\usepackage[T1,T2A]{fontenc}
\usepackage[utf8]{inputenc}
\usepackage{url}
\usepackage{amsfonts}
\usepackage{amssymb, amsthm}
\usepackage{amsmath}
\usepackage{mathtools}
\usepackage{needspace}
\usepackage[pdftex]{graphicx}
\usepackage{hyperref}
\usepackage{datetime}
\usepackage{epigraph}
\usepackage{verbatim}
\usepackage{mathtools}
\usepackage{xcolor}
\usepackage{tikz}
\usetikzlibrary {matrix}

\usepackage{tikz}
\usetikzlibrary{calc}
\usepackage{nicematrix}

\numberwithin{equation}{section}

\newcommand{\Z}{\mathbb{Z}}

\newcommand{\N}{\mathbb{N}}
\newcommand{\R}{\mathbb{R}}
\newcommand{\Q}{\mathbb{Q}}

\newcommand{\Cm}{\mathbb{C}}

\newcommand{\eps}{\varepsilon}

\DeclareMathOperator{\supp}{supp} 

\DeclareMathOperator{\beqq}{\begin{equation}} 

\DeclareMathOperator{\eeqq}{\end{equation}}

\renewcommand{\phi}{\varphi}

\newcommand{\beq}{\begin{equation}}
\newcommand{\eeq}{\end{equation}}

\newtheorem{Thm}{Theorem}[section]
\newtheorem{theorem}[Thm]{Theorem}

\newtheorem{lemma}[Thm]{Lemma}

\newtheorem{remark}[Thm]{Remark}

\usepackage{tikz}
\usetikzlibrary{matrix,decorations.pathreplacing}
\pgfkeys{tikz/mymatrixenv/.style={decoration=brace,every left delimiter/.style={xshift=4pt},every right delimiter/.style={xshift=-4pt}}}
\pgfkeys{tikz/mymatrix/.style={matrix of math nodes,left delimiter=(,right delimiter={)},inner sep=1pt,row sep=0em,column sep=1em,nodes={inner sep=6pt}}}

\begin{document}

\sloppy
\title[Frames for compactly supported functions with irrational density]
{Frames for compactly supported functions with irrational density}

\author{Yurii Belov}
\address{Yurii Belov,
\newline Department of Mathematics and Computer Science
St. Petersburg State University,
14th Line 29b, Vasilyevsky Island, St. Petersburg, Russia, 199178,
\newline {\tt j\_b\_juri\_belov@mail.ru} }
\author{Aleksei Kulikov}
\address{Aleksei Kulikov,
\newline University of Copenhagen, Department of Mathematical Sciences,
Universitetsparken 5, 2100 Copenhagen, Denmark,
\newline {\tt lyosha.kulikov@mail.ru}  }

\begin{abstract} 
We find sufficient conditions on a compactly supported function $g$, $\supp g = [a,b]$ which guarantee that the Gabor system $$\mathcal{G}(g;\alpha,\beta)=\{e^{2\pi i \beta m x}g(x-\alpha n)\}_{m,n\in\mathbb{Z}}$$ is a frame for all $\alpha < b-a, \alpha\beta < 1, \alpha\beta \notin\Q$. These conditions are on one hand satisfied by almost all such functions, and on the other hand are explicit enough that we can give many concrete examples of the functions $g$ which give us a frame e.g. $g(x) = \exp(\frac{1}{x^4-1})\chi_{(-1,1)}(x)$.
\end{abstract}

\maketitle
\section{Introduction}
For a function $g\in L^2(\R)$ and the numbers $\alpha, \beta > 0$ we consider the Gabor system
 $$\mathcal{G}(g;\alpha,\beta):=\{g_{m,n}\}_{m,n\in\mathbb{Z}}=\{e^{2\pi i \beta m x}g(x-\alpha n)\}_{m,n\in\mathbb{Z}}$$
and ask when it is a frame, that is when there are constants $A, B > 0$ such that for all $h\in L^2(\R)$ we have
$$A\|h\|^2_2\leq\sum_{m,n\in\mathbb{Z}}|(h,g_{m,n})|^2\leq B\|h\|^2_2.$$
 One well-known necessary condition for this is $\alpha\beta \le 1$ \cite[Corollary 7.5.1]{Gro}, and the inequality is strict if the function $g$ is sufficiently smooth and decays fast enough \cite[Theorem 6.5.1]{Gro}. If the function $g$ is compactly supported i.e. $g(x) = 0, x\notin [a, b]$, then we also have a trivial necessary condition $\alpha \le b-a$ and if $\alpha = b-a$ then the system is a frame if and only if $\alpha\beta \le 1$ and $g, g^{-1}\in L^\infty ((a, b))$. 
 
 We show that if $g(x) \ne 0$ for $a < x < b$ and function $g$ is generic enough then these conditions are also sufficient for the system to be a frame if $\alpha\beta\notin \Q$. Notably, all previous results except recent preprint \cite{BK} concerned either some concrete functions, specific families of functions (rational functions, total positive functions, one-sided and two-sided exponentials, hyperbolic secant, Gaussian function, shifted sinc-function, Haar function, characteristic function of an interval, see \cite{BB, BKL, BS, DS, DZ, Jans2, JansStr, Jans, GroTP, Gro1, Gro2,L,S,SW}) or required strong restrictions on pairs $(\alpha,\beta)$ (see e.g. \cite{Faul, GL, GS}). Our results are to the best of our knowledge the first which apply to a wide range of functions while describing almost all of the frame set. For example, we prove the following theorem.
\begin{theorem}\label{main}
Let $f:\Cm\backslash A\to \Cm$ be an analytic function, where $A$ is some locally finite subset of $\Cm$. Assume that for some $a < b$ we have $A\cap (a, b)=\varnothing$, $f\in L^\infty((a, b))$ and $f(x) \ne 0, x\in (a, b)$, and that there exists $t\in \R$ such that on the horizontal line ${\rm{Im}}\,\, z = t$ there is exactly one essential singularity of the function $f$. Put $g(x) = f(x) \chi_{(a, b)}(x)$. Then for all $\alpha, \beta > 0$ such that $\alpha < b-a, \alpha\beta < 1, \alpha\beta \notin \Q$ the system $\mathcal{G}(g;\alpha,\beta)$ is a frame.
\end{theorem}

Prime example of the application of this theorem to have in mind is the function $f(z) = \exp(\frac{1}{z^4-1})$ with $a = -1, b = 1$. This function is obviously non-zero on $(-1, 1)$ and on the line ${\rm{Im}}\,\, z = 1$ it has exactly one essential singularity.

Windows from Theorem \ref{main} can belong to the Balian--Low class, see \cite[Section 8.4]{Gro}. Moreover, for any $s<1$ there exists a window $g$ satisfying the assumptions of Theorem \ref{main} such that $|\hat{g}(\xi)|\leq C\exp(-|\xi|^s)$, for example $g(x) = \exp({-(1-x^4)^{-N}})\chi_{(-1, 1)}(x)$ for $N > \frac{1}{1-s}$. Indeed, one can check that $|g^{(m)}(x)|\leq C^{m}(m!)^t$ for any $m\in\mathbb{N}$ and $t > 1 + \frac{1}{N}$. Hence, function $g$ belongs to the corresponding Gevrey class and we have $|\hat{g}(\xi)|\leq C\exp(-|\xi|^{s})$ for every $s < 1-\frac{1}{N}$. 

\medskip

While analyticity is a convenient way to get explicit examples from our methods, it is by no means necessary. In fact, our approach works for a generic or random function $g$. For example, we also have the following theorem.
\begin{theorem}
Let $h:\R\times\R\to \R$ be a smooth function vanishing outside of the triangle $0 < t < x < 1$ and positive inside of this triangle. Let $B:\R\to\Cm$ be a complex-valued Brownian motion starting at $1$, $B(0)=1$ with variance $1$.  Then almost surely for the function $g$ defined as $$g(x) = \int_0^x B(t)h(x,t)dt,$$ the system $\mathcal{G}(g; \alpha, \beta)$ is a frame for all $\alpha\beta < 1, \alpha\beta\notin\Q, \alpha < 1$.
\label{bth}
\end{theorem}
The function $g$ is almost surely $C^\infty(\R)$ and so we get a plethora of functions from the Balian--Low class giving us a frame for all admissible parameters with irrational $\alpha\beta$. 

\begin{remark}
We consider the complex-valued Brownian motion only to guarantee that almost surely for all $x\in (0, 1)$ we have $g(x)\neq 0$. For complex-valued curves the probability of hitting $0$ is $0$, but if the function was real-valued it would be non-zero with positive probability strictly less than $1$. For the same reason we assumed that $h$ is positive for $0 < t < x < 1$, this can also be replaced with some generality assumption. 
\end{remark}
\medskip

 The most prominent way to construct non-frame Gabor systems with $\alpha\beta<1$ is to fix a rational $\alpha\beta$ and use matrix Zak transform criterion (see \cite[Theorem 8.3.3]{Gro}).  In particular, Lyubarskii and Nes in \cite{LN} discovered that if $g(x)=-g(-x)$ and $g$ belongs to the Feichtinger algebra, then systems $\mathcal{G}(g;1,n\slash (n+1))$ are non-frame for all $n\in\mathbb{N}$. For example, the system $\mathcal{G}(xe^{-x^2};1,1\slash2)$ is not a frame. So in a positive direction it is natural to consider irrational densities $\alpha\beta$.      

%On the other hand, for generic rational function and irrational $\alpha\beta$, $\alpha\beta<1$ we always have a frame property (see cite).

\subsection{Strategy of the proof} We begin by recalling the Ron--Shen criterion \cite{RonShen} for the Gabor system to be a frame (see e.g. \cite[Proposition 6.3.4]{Gro} for the form that we use).
\begin{theorem}
Let $g\in L^2(\R)$ and consider the infinite matrix $$G(x) = \left\{g\left(x - \alpha n + \frac{1}{\beta}m\right)\right\}_{n,m\in\mathbb{Z}}.$$ The system $\mathcal{G}(g;\alpha,\beta)$ is a frame if and only if for some $A, B > 0$ and for almost all $x\in \R$ and all $v\in \ell^2(\Z)$ we have
\begin{equation}
A||v||_2 \le ||G(x) v||_2 \le B||v||_2.
\label{meq}
\end{equation}
\label{RS}
\end{theorem}

Note that if the function $g$ is bounded and compactly supported then the upper bound is always satisfied so we will be focusing on the lower bound.

The proofs of our results are based on the following three ideas: first, if we remove a row from the infinite matrix $G(x)$ then it would be only harder to prove the lower bound. Second, if after the removal of some rows the matrix $\tilde{G}(x)$ we get will be an orthogonal sum of finite-dimensional square matrices $G_k(x)$, $$\tilde{G}(x) = \bigoplus\limits_{k\in\Z} G_k(x),$$ 
$$
\newcommand*{\tempb}{\multicolumn{1}{|c}{0}}
\newcommand*{\tempv}{\multicolumn{1}{|c}{\vdots}}
\newcommand*{\temps}{\multicolumn{1}{|c}{\star}}
\newcommand*{\tempd}{\multicolumn{1}{|c}{\ddots}}
\begin{array}{c}
\begin{pmatrix} 
\star & \star & \ldots & \star & \tempb &  \ldots  & 0&\ldots  &0\\
\star & \star & \ldots & \star  & \tempb  & \ldots & 0 & \ldots  & 0\\
\vdots & \vdots&  & \vdots &   \tempv& \vdots  &  \vdots  & & \vdots \\
\star & \star & \ldots & \star   & \tempb & \ldots &  0 & \ldots  &0 \\
\cline{1-7}
0 & 0 & \ldots & 0 & \temps & \ldots  & \star & \tempv & 0    \\
0 & 0 & \ldots & 0 & \temps & \ldots & \star  & \tempv & 0   \\
\vdots & &   & \ddots &  \tempv & \ddots & &\tempd & \ddots&   \\
\cline{5-9}
0 & 0 & \ldots   & 0 & 0 &   \ldots& 0 & \temps & \star \\
0 & 0 & \ldots   & 0 & 0 &  \ldots & 0 &\temps & \star\\
\end{pmatrix},\\
\text{Square blocks } G_k(x),
\end{array}
$$

then the lower bound for $\tilde{G}(x)$ holds if and only if it holds uniformly for each of the matrices $G_k(x)$. Third, for finite-dimensional matrices of uniformly bounded sizes and with uniformly bounded entries, we have a uniform lower bound if and only if their determinants are uniformly bounded away from zero.

\medskip

This approach is similar to the proof of Theorem 1.3 from \cite{BKL}. Just like in that paper, we will, using the irrationality of $\alpha\beta$, get from one non-zero determinant for one $x$ to such determinants for all $x\in \R$ with uniformly bounded sizes. However, in \cite{BKL} the proof that at least one of the determinants is non-vanishing was technically the hardest step, as we had to do some lengthy computations with these determinants and rely on some slight algebraic miracles. On the other hand, in the present case, proof of the non-vanishing of some determinant is not much more difficult than the other parts of the proof. We think there are two reasons for this: first of all, rational functions that we studied in \cite{BKL} are not generic enough, which makes showing that there is no hidden conspiracy among these matrices harder. Second, unlike \cite{BKL} here we were able to use irrationality of $\alpha\beta$ again in the third step, which made proving non-vanishing of the determinant even easier. 

\medskip
%The proof of our results splits into three steps.

%\begin{enumerate}
%\item Observe finite-dimensional matrices inside of the operators $G_x$ and reduce the problem to showing that the determinants of these matrices are everywhere non-zero,
%\item Notice almost periodicity of these matrices, and use density of the irrational rotation and the assumption of non-vanishing of the function $g$ to get that it is enough to show that the determinants are non-zero for only one value of $x$,
%\item Show that the determinant function is not identically zero.
%\end{enumerate}
%This strategy is similar to the proof of Theorem 1.3 from \cite{7675}. Notice, however, that in \cite{7675} the third step was technically the hardest one, we had to do some lengthy computations with these determinants and rely on some slight algebraic miracles. On the contrary, here the third step is not much harder than the first two. We think there are two reasons for it: first of all rational functions that we studied in \cite{7675} are not generic enough, which makes showing that there is no hidden conspiracy among these matrices harder. Secondly, unlike \cite{7675} here we were able to use irrationality of $\alpha\beta$ again in the third step, which made proving non-vanishing of the determinant even easier. 

In our setting, these ideas impose the following conditions on the function $g$ and $\alpha, \beta$:
\begin{enumerate}
\item If the function $g$ is compactly supported then each row of the matrix $G(x)$ has only finitely many non-zero terms and if in addition $\alpha\beta < 1$ then we can find a finite set of rows such that the matrix formed by them is square, so we can look at its determinant.
\item If $g(x)\neq 0, a < x < b$ and $\alpha < b-a$ then we can shift from any such square matrix to any other inside of the infinite matrix $G(x)$, and if $\alpha\beta \notin \Q$ then in a uniformly bounded number of steps we can get close to any other such square submatrix.
\item If the function $g$ is generic then the determinant of such square submatrix as a function of $x$ will not be identically zero, hence we can always get to a submatrix with non-zero determinant. Moreover, if $\alpha\beta\notin \Q$ then the elements appearing in the matrix $G(x)$ are all $g$ at different arguments. So, to prove that the determinant does not vanish completely for a generic function $g$, it is basically enough to show that there is at least one permutation in the expansion of the determinant giving us non-zero contribution. The permutation that we will consider will actually be an identity permutation.
\end{enumerate}

Technical conditions on the function $g$ which are sufficient for the sytem $\mathcal{G}(g;\alpha,\beta)$ to be a frame are somewhat cumbersome, so we will not present them in the introduction, see Theorem \ref{framecon} for the exact statement. It is important to note that these conditions are satisfied by a generic function $g$ for pretty much any definition of generic, see e.g. Theorem \ref{bth}. Alternatively, they are satisfied for  analytic functions satisfying some specific conditions like in Theorem \ref{main}.

\section{Structure of the Ron--Shen matrices and their finite-dimensional submatrices}\label{prep}
Recall that the Ron--Shen matrix $G(x)$ is defined by $\{g(x-\alpha n+\frac{1}{\beta}m)\}_{n,m\in\mathbb{Z}}$. Note that we have 
$g(x-\alpha n+\frac{1}{\beta}m)=0$ if $x-\alpha n+\frac{1}{\beta}m\not\in(a,b)$. So, in each row there are at most $\beta(b-a)+1$ non-zero elements and the places for potential non-zero elements form a segment, and these segments move to the right as $n$ increases. Specifically, we will call a pair $(n, m)\in \Z^2$ good if $x-\alpha n+\frac{1}{\beta}m\in(a,b)$. We will use the following purely combinatorial lemma.
\begin{lemma}
We have
\begin{enumerate}
    \item If $\alpha\beta < 1$ and $(n, m)$ and $(n, m+1)$ are good pairs then $(n+1, m+1)$ is also a good pair.
    \item If $\alpha\beta < 1$ and $\alpha < b - a$ then for each $m\in \Z$ there exists $n\in\mathbb{Z}$ such that $(n, m)$ is a good pair, $(n, m+1)$ is not a good pair and $x-\alpha n+\frac{1}{\beta}m\in [a+\eps, b-\eps]$ where $\eps = \frac{1}{2}\min(b-a-\alpha, \alpha, \frac{1}{\beta}-\alpha)$.
\end{enumerate}
\label{l21}
\end{lemma}
\begin{proof}(i): Since $\alpha - \frac{1}{\beta} <0<\alpha$ we have $$a<x-\alpha n+\frac{1}{\beta}m< x - \alpha(n+1)+\frac{1}{\beta}(m+1) < x-\alpha n + \frac{1}{\beta}(m+1) < b.$$
Hence, $(n+1, m+1)$ is good.\\
(ii): Since $\alpha<b-a$ for any $m$ there exists at least one good pair $(n,m)$. Let $n$ be a minimal integer such that pair $(n,m)$ is good. Since $(n-1, m)$ is not good we have $x-\alpha(n-1) + \frac{1}{\beta}m \ge b$. Therefore, $x-\alpha n + \frac{1}{\beta}(m+1)\ge b$ because $0 < \frac{1}{\beta}-\alpha$. Thus, $(n, m+1)$ is not good.

If $x-\alpha n + \frac{1}{\beta}m \le a + \eps$ then $(n-1, m)$ would be a good pair which is not allowed. If $x-\alpha n + \frac{1}{\beta}m > b - \eps$ then we will instead look at the pair $(n+1, m)$. It is still a good pair but $(n+1, m+1)$ is still not a good pair. On the other hand, $x-\alpha (n+1) + \frac{1}{\beta}m \in [a+\eps, b-\eps]$.
\end{proof}

Let us consider any row $n$ with at least one good pair. Let $(n, m)$ be the first good pair in this row. Let $l$ be the biggest integer such that $x-\alpha(n+l) + \frac{1}{\beta}(m+l) < b$. Such an $l$ exists since $\frac{1}{\beta} - \alpha > 0$, and moreover it is uniformly bounded $l\le l(\alpha, \beta, a, b)$. Note that $l\ge 0$ and $(n+l, m+l)$ is good. From Lemma 2.1 (i) it follows that $(n+l, m+l+1)$ is not good. Consider the submatrix formed by the rows from $n$ to $n+l$. Note that in this submatrix all good pairs are in the square submatrix formed by the rows from $n$ to $n+l$ and columns from $m$ to $m+l$. Finally, all pairs $(n+k, m+k), 0\le k\le l$ are good.

For each $x\in [0, \alpha]$ we fix a good pair $(n_x, m_x)$ which is the first in its row. Note that we can do so with a uniform bound $|n_x|, |m_x|\le n(\alpha, \beta, a, b)$ by e.g. Lemma \ref{l21} (ii). By $M_x$ we denote the submatrix considered in the previous paragraph.

$M_x=$\begin{tikzpicture}[baseline=0cm,mymatrixenv]
    \matrix [mymatrix,text width=1.1em,align=center] (m)  
    {
\star & \star & \ldots & \star  &  0  & 0&0&\ldots  &\ldots &\ldots\\
\star & \star & \ldots & \star    & 0 & 0 &0& \ldots  & \ldots&\ldots\\
0 & \star & \star & \ldots & \star  & 0 & 0 & \ldots  &\ldots &\ldots \\
0 & 0 & \star & \star & \ldots & \star  & 0 & \ldots  &\ldots&\ldots\\
\vdots & &   & \ddots &   \ddots & & \ddots& &\vdots  &\vdots\\
0 & 0 & \ldots & 0 &  \star &\star & \ldots& \star  &\ldots&\ldots\\
0 & 0 & \ldots & 0  & \star &\star & \ldots  & \star &0&\ldots\\
0 & 0 & \ldots & 0 & 0 &\star & \ldots & \star & \star&\ldots\\
0 & 0 & \ldots & 0 & 0 & 0 & \star & \ldots & \star&\ldots\\
\vdots & &   & \ddots &   \ddots & & \ddots& &\vdots & \vdots \\
0 & &   & \ddots &   \ddots & & \ddots &\star&\ldots &\star \\
0 & &   & \ddots &   \ddots & & \ddots &\star &\ldots &\star \\
    };
    %first is horizontal from the left, second is vertical from the bottom
    \draw (-6.1, 3.4)--(-6.1, 4.55);
    \draw (-6.1, 3.4)--(-1.6, 3.4);
    \draw(-1.6, 3.4)--(-1.6, 4.55);
    \draw(-1.6, 4.55)--(-6.1, 4.55);
    
    \draw (-0.8, -0.35)--(-0.8, 0.8);
    \draw (-0.8, -0.35)--(3.4, -0.35);
    \draw (3.4, -0.35)--(3.4, 0.8);
    \draw(3.4, 0.8)--(-0.8, 0.8);

    \draw (2.8, -3.2)--(2.8, -4.6);
    \draw (2.8, -3.2)--(6.1,-3.2);
    \draw(6.1,-3.2)--(6.1,-4.6);
    \draw(6.1,-4.6)--(2.8,-4.6);
\end{tikzpicture}

In the picture we boxed some of the pairs of rows where we get a catch-up of the column count to the row count because of lack of shift by $1$ to the right in the bottom row. When the column count becomes equal to the row count we get the square matrix.  Now, we are ready to state a sufficient condition for the system $\mathcal{G}(g;\alpha,\beta)$ to be a frame.

\begin{theorem}\label{framecon}
Assume that $\alpha\beta < 1$, $\alpha\beta\notin\Q$, $\alpha < b-a$, $g\in L^\infty(\R)$, $g(x) = 0, x\notin (a, b)$, $g^{-1}\in L^\infty([a+\eps, b-\eps])$, where $\eps = \frac{1}{2}\min(b-a-\alpha, \alpha, \frac{1}{\beta}-\alpha)$ and that there exists an open interval $I\subset (0, \alpha)$ and $\delta > 0$ such that $|\det (M_x)|\ge \delta$ for $x\in I$. Then $\mathcal{G}(g;\alpha,\beta)$ is a frame. 
\label{framecon}
\end{theorem}
\begin{proof} We will use the Ron--Shen criterion (Theorem \ref{RS}). The upper estimate follows from the inclusion $g\in L^\infty(\R)$. So, we will focus only on  the lower inequality. For any $M,N\in\Z$ the Ron--Shen matrix $G(x-N\alpha+M\frac{1}{\beta})$ is the shift of the Ron--Shen matrix $G(x)$. By irrationality of $\alpha\beta$ we can without loss of generality assume that $x\in I$. 

Consider the square matrix $M_x$. We claim that it is invertible and that its inverse has uniformly bounded entries. Indeed, we have
$$M_x^{-1} = \det(M_x)^{-1}{\rm{adj}}(M_x).$$
Since $g\in L^\infty(\R)$ and the size $l_x$ of $M_x$ is uniformly bounded, the entries of ${\rm{adj}}(M_x)$ are uniformly bounded. On the other hand, since by assumption $|\det(M_x)|\ge \delta$, the inverse of the determinant is also uniformly bounded and the claim follows.

Next, we take the smallest $\tilde{n} > n_x+l_x, \tilde{m} > m_x+l_x$ such that $x-\tilde{n}\alpha + \frac{1}{\beta}\tilde{m}\in I$. Again, by irrationality of $\alpha\beta$ this is possible, and, moreover, we can have a uniform upper bound on $\tilde{n}, \tilde{m}$. Once we get to this good pair $(\tilde{n}, \tilde{m})$ we will encounter a matrix $M_{\tilde{x}}, \tilde{x} = x-\alpha \tilde{n} + \frac{1}{\beta}\tilde{m}$ which we know also has a lower bound on its determinant, and so it is invertible with a bounded inverse. 

For each $m\in (m_x + l_x, \tilde{m})$ we consider the row given by Lemma \ref{l21} (ii).  We will consider the matrix $\tilde{M}_x$ which is obtained from the matrix $M$ by adding all these rows. All other rows with indices between $n_x+l_x$ and $\tilde{n}$ we will throw away from the matrix $G(x)$ because, as we explained in the introduction, doing this will make our goal only harder.

$$
\newcommand*{\tempb}{\multicolumn{1}{|c}{0}}
\newcommand*{\tempv}{\multicolumn{1}{|c}{\vdots}}
\newcommand*{\temps}{\multicolumn{1}{|c}{\star}}
\newcommand*{\tempd}{\multicolumn{1}{|c}{\ddots}}
\tilde{M}_x=
\begin{pmatrix} 
\star & \star & \ldots & 0 & \tempb &  0  & 0&\ldots  &0\\
\star & \star & \ldots & 0  & \tempb  & 0 & 0 & \ldots  & 0\\
\vdots & \vdots&  & \vdots &   \tempv& \vdots  &  \vdots  & & \vdots \\
0 & 0 & \ldots & \star   & \tempb & 0 &  0 & \ldots  &0 \\
\cline{1-4}
0 & 0 & \star & \ldots & \star  & 0 & 0 & \ldots  &0 \\
0 & 0 & 0 & \star & \ldots & \star  & 0 & \ldots  &0\\
\vdots & &   & \ddots &   \ddots & & \ddots& &\vdots  \\
0 & \ldots  & 0 & 0 & 0 &  \star  & \ldots& \star & 0\\
0 & \ldots  & 0 & 0 & 0 & 0 & \star  & \ldots& \star\\
\end{pmatrix}.$$

Note that the size of $\tilde{M}_x$ is still uniformly bounded and, since it has a block-diagonal structure, its determinant is equal to $\det(M_x)$ multiplied by the values of $g$ at the pairs $(n, m)$ given by Lemma \ref{l21} (ii). Since all such pairs satisfy $x-\alpha n + \frac{1}{\beta} m \in [a+\eps, b-\eps]$, by the assumption $g^{-1}\in L^\infty ([a+\eps, b-\eps])$ they are uniformly bounded away from zero. So, the determinant of $\tilde{M}_x$ is also uniformly bounded away from zero and by the same reasoning it is invertible and its inverse has uniformly bounded entries.

By doing the same procedure after $M_{\tilde{x}}$, and also going up and left from the matrix $M_x$, we can throw away some rows from the Ron--Shen matrix $G(x)$ so that the resulting matrix can be decomposed into an orthogonal sum of uniformly bounded in size and  invertible with uniformly bounded inverses matrices. Orthogonal sum of operators is bounded from below if and only if each one of them is bounded from below with a uniform lower bound. It remains to note that finite-dimensional matrices of uniformly bounded sizes are bounded from below if and only if all entries of their inverses are bounded from above, just as we showed.
\end{proof}

%Among the assumptions of Theorem \ref{framecon}, all the assumptions except for the last one are easy to verify for a given function $g$ and numbers $\alpha, \beta, a, b$, in particular the assumptions about the function $g$ are true as long as $g$ is continuous and non-vanishing on $(a, b)$. The only complicated one is the assumption about $\det(M_x)$ which requires additional ideas.
%There exists positive integer $l$ such that pair $(n+l-1,m+l-1)$ is good but pair $(n+l,m+l)$ is not good. 

Among the conditions in Theorem \ref{framecon}, the only two which are generally interesting and not always satisfied in practice are $g^{-1}\in L^\infty([a+\eps, b-\eps])$ and the assumption about $\det(M_x)$. The first one of them is easy to verify for a given function $g$ and in general it is true for all $\alpha, \beta$ with $\alpha\beta < 1, \alpha < b-a$ as long as $g$ is continuous and non-vanishing on $(a, b)$. Plus, if $g$ is a generic complex-valued function then the curve $g([a, b])$ almost surely does not hit $0$ (for real-valued functions, the probability is usually positive but strictly less than $1$).

On the other hand, the condition about $\det(M_x)$ is much more cumbersome and peculiar. For example, it is not satisfied for the function $g(x) = \chi_{[0, 1]}(x)$ since there exists a pair $(\alpha, \beta)$ with $\alpha\beta < 1, \alpha<1, \alpha\beta\notin\Q$ such that $\mathcal{G}(g;\alpha,\beta)$ is not a frame, see \cite{DS}. But, since for continuous $g$ this condition basically means that the determinant is non-zero somewhere, its negation is a system of equations parametrized by all $x\in I$. So, for this condition to not be satisfied, function $g$ has to satisfy an infinite (in fact, uncountable) family of equations simultaneously, and this is very natural to exclude via complex analysis or randomness.

\section{Not identically vanishing determinants}

By the discussion in the previous section, for the proof of Theorem \ref{main} it remains to establish the non-vanishing of $\det(M_x)$. Specifically, we will prove the following result.

\begin{theorem}\label{maindet}
Let $f:\Cm\backslash A\to \Cm$ be an analytic function, where $A$ is some locally finite subset of $\Cm$. Assume that for some $a < b$ we have $A\cap (a, b)=\varnothing$ and that there exists $t\in \R$ such that on the line ${\rm{Im}}\,\, z = t$ there is exactly one essential singularity of the function $f$. Put $g(x) = f(x) \chi_{(a, b)}(x)$. Then for all $\alpha, \beta > 0$ such that $\alpha < b-a, \alpha\beta < 1, \alpha\beta \notin \Q$ there exists an interval $I\subset(0,\alpha)$ such that $|\det (M_x)|>\delta$ for $x\in I$ and some $\delta > 0$.
\end{theorem}
\begin{proof}
We begin by noting that $M_x$ can have only finitely many "combinatorial structures" $ $ and that said structure is a piecewise-continuous function on the interval $(0, \alpha)$. By combinatorial structure we mean the size of the matrix $M_x$ (which, as we discussed in the previous section, is uniformly bounded) and the set of good pairs $(n, m)$ inside of it. Since they are defined by finitely many inequalities on $x\in (0, \alpha)$, we can find an interval $J\subset (0, \alpha)$ on which the combinatorial structure is constant. We will consider $x$ only from this interval $J$.

On this interval $M_x$ coincides with an analytic matrix-valued function, and hence its determinant $D(x)=\det(M_x)$ is an analytic function. It can be analytically extended to $\Cm\backslash B$ where $B$ is a finite union of real translations of the set $A$, where translations are by the numbers $\alpha n - \frac{1}{\beta}m$. In particular, since $A$ is locally finite, $B\cap J$ has at most finitely many points. Because $D(x)$ is continuous on $\Cm\backslash B$, to find an interval $I$ on which it is bounded away from $0$ it is enough to show that $D$ does not vanish identically.

Note that all numbers $x-\alpha n+\frac{1}{\beta}m$, $n,m\in\mathbb{Z}$ are different. So, all good entries of matrix $M_x$ have essential singularities at different points on the line ${\rm{Im}}\,\, z = t$. Let $D_k(x)$ be a determinant of the top-left $k\times k$ submatrix of the matrix $M_x$. We will show by induction on $k$ that $D_k(x)$ is not identically vanishing. For the basis of induction we just note that, by construction of $M_x$, its top-left entry is non-zero.

For $D_{k+1}(x)$ we will use the Laplace expansion. We get $D_k(x)$ times the element of $M_x$ at the position $(k+1, k+1)$ plus some other terms which do not involve this element. Note that this pair is good by the construction of the matrix $M_x$. We claim that the result will have an essential singularity at the same point on the line ${\rm{Im}}\, z = t$ as this element. Indeed, from the previous paragraph we know that all other terms and $D_k(x)$ will not have an essential singularity at this point. On the other hand $D_k(x)$ does not vanish identically, so at the point of the essential singularity it can vanish only to a finite multiplicity. When we multiply an essential singularity by something vanishing only to a finite multiplicity at this point, we can not destroy the essential singularity. Hence, $D_{k+1}(x)$ has an essential singularity at this point and therefore can not vanish identically. 
\end{proof}

%{\color{red}
%\begin{proposition}
%    For any $x_1, x_2, \ldots , x_l\in (0, 1)$ the distribution of $$(g_k(x_1), g_k(x_2),\ldots , g_k(x_l))$$ is absolutely continuous with respect to the Lebesgue measure on $\Cm^l$.
%\end{proposition}}

Now, we turn to the proof of Theorem \ref{bth}. We will again use the sufficient condition from Theorem \ref{framecon}. Clearly, almost surely $g\in C^1(\R)$ since $B$ is almost surely continuous. First, we will show that almost surely $g$ does not vanish in the open interval $(0, 1)$, since $g$ is continuous, this would be enough to get $g^{-1}\in L^\infty (\eps, 1-\eps)$ for any $\eps > 0$. We will use the following lemma about the distribution of integrals of Brownian motions.
\begin{lemma}\label{brown lemma}
Let $B:[0, t]\to \R$ be a real-valued Brownian motion with $B(0)=r$ and variance $1$ and let $u:[0,t]\to\R$ be a deterministic smooth function. Then the random variable
$$\int_0^t B(x)u(x)dx$$
is Gaussian with expected value $r\int_0^tu(x)dx$ and variance $\int_0^t \left(\int_0^r u(x)dx\right)^2dr$.
\end{lemma}

To show that $g$ almost surely does not vanish on $(0, 1)$ we will show that it almost surely does not vanish on $(\frac{1}{n},1-\frac{1}{n})$ for all $n\in\N$. Since $B$ is almost surely continuous, $g$ is almost surely Lipschitz, so we can assume that $|g(x)-g(y)|\le 
\kappa|x-y|$ for some $\kappa\in\N$ (and then take a countable union over $\kappa\in\N$). Pick a big natural number $k$ and consider the numbers $\theta_{l,k}=\frac{l}{k}$ with $\frac{1}{n}<\theta_{l,k}<1-\frac{1}{n}$. If $g(x) = 0$ for some $x\in (\frac{1}{n},1-\frac{1}{n})$ then for some $l$ we have $|g(\theta_{l,k})|\le \frac{\kappa
}{k}$. Recall that for each $l, k$ the distributions of real and imaginary parts of $g(\theta_{l,k})$ are independent real Gaussian random variables with variance uniformly bounded away from $0$ (by our assumptions on function $h$ and $\frac{1}{n}<\theta_{l,k}<1-\frac{1}{n}$). Hence the probability of the event $|g(\theta_{l,k})|\le \frac{\kappa}{k}$ is $O(\frac{1}{k^2})$. Therefore the probability of the union of these events over all $l$ is $O(\frac{1}{k})$ which tends to $0$ as $k\to \infty$, so almost surely $g$ has no zeroes on $(\frac{1}{n},1-\frac{1}{n})$.
\medskip

Next, we go to the determinant $\det(M_x)$. First of all, we will only consider the matrices $M_x$ of the given combinatorial structure, as there are only countably many different combinatorial structures. Second, since, for a given combinatorial structure, the determinant is almost surely a continuous function, just as in the proof of Theorem \ref{maindet} we will just show that it is almost surely not zero at at least one point.

\medskip

In the matrix $M_x$ with the given combinatorial structure, the entries are $g(x - \alpha n +\frac{1}{\beta}m)$ for some finite set of pairs $(n, m)\in \Z^2$. If $\alpha\beta$ is not equal to any ratio $\frac{m}{n}$ of any of such pairs (in particular, if $\alpha\beta\notin\Q$), then the arguments of the entries are all different. We want the arguments to be quantitatively bounded away from each other and also to have a fixed order among themselves, which we can do by doing another countable subdivision of the space of pairs $(\alpha, \beta)$ depending on the minimal distance from $\alpha\beta$ to any of these (importantly, finitely many) rational numbers. 

Now, we will treat the numbers $-\alpha n + \frac{1}{\beta}m$ as arbitrary real numbers $s_l$,  forgetting any possible dependencies between them and only caring about their order and the minimal distance between any two of them. In this way, each $s_l$ runs in the interval $[u_l, v_l]$ with $u_{l+1}-v_l \ge \eps$ for some $\eps>0$ and the determinant is 
$$\det M_x = Q(g(x+s_1), g(x+s_2),\ldots , g(x+s_L))$$
for some multilinear polynomial $Q$ and we want to show that almost surely for all $s_l\in [u_l, v_l]$ at least for some $x$ in a small interval $I$ it is non-zero. Importantly, $Q$ is not identically zero, since, as was shown in Section \ref{prep}, the main diagonal of $M_x$ necessarily contains admissible arguments for $g$. 

Let $g(x+s_{m_1})$ be the entry with the biggest index which non-trivially appears in $\det M_x$. Using multilinearity of $\det M_x$ we can write
$$\det M_x = g(x+s_{m_1})Q_1(g(x+s_1),\ldots , g(x+s_{m_1-1}))+R_1(g(x+s_1),\ldots , g(x+s_{m_1-1}))$$
for some multilinear polynomials $Q_1, R_1$. Let us continue the same procedure for $Q_1$, taking the biggest value $x+s_t$ which appears in it and separating the corresponding value of $g(x+s_{m_2})$. In the end we will get a decreasing sequence of integers $m_1, m_2, \ldots , m_k$, and two sequences of multilinear polynomials $Q_1, \ldots , Q_k$ and $R_1, \ldots , R_k$ such that $Q_l$ and $R_l$ depend only on the variables up to $g(x+s_{m_l-1})$ (in particular, $Q_k$ and $R_k$ are constants with $Q_k$ being non-zero) and such that
\begin{multline*}Q_l(g(x+s_1),\ldots , g(x+s_{m_l-1})) =\\
g(x+s_{m_{l+1}})Q_{l+1}(g(x+s_1),\ldots,g(x+s_{m_{l+1}-1})+R_{l+1}(g(x+s_1),\ldots,g(x+s_{m_{l+1}-1}),\end{multline*}
where we set $Q_0=Q$, $R_0=0$, $m_0=L+1$. Without loss of generality we can assume that the interval $I=[c,d]$ of allowed $x$'s has length less than $\frac{\eps}{3}$ so that all the possible values of $x+s_t$ do not intersect as well. We are going to show by induction on $l$ that almost surely for all $s_1\in [u_1, v_1],\ldots , s_{m_{k-l+1}-1}\in [u_{m_{k-l+1}-1}, v_{m_{k-l+1}-1}]$ there exists $x\in I$ such that $$Q_{k-l+1}((g(x+s_1),\ldots , g(x+s_{m_{k-l+1}-1}))$$ is non-zero. The base case $l = 1$ is clear as then $Q_k$ is a non-zero constant. For the induction step, first of all we notice that 
$Q_{k-l+1}((g(x+s_1),\ldots , g(x+s_{m_{k-l+1}-1}))$ 
depends only on the distribution of $B$ up to the time $d+v_{m_{k-l+1}-1}$, which is strictly less than $c+u_{m_{k-l+1}}$ which is the earliest point where $g(x+s_{m_{k-l+1}})$ can appear. So, to begin with we will condition on the value of $B(t)$ up to the time $d+v_{m_{k-l+1}-1}$ so that $B(t)$ is continuous and the above assumption about the existence of $x\in I$ is satisfied.  

Since $Q_{k-l+1}((g(x+s_1),\ldots , g(x+s_{m_{k-l+1}-1}))$ is clearly continuous in $x$, if it is non-zero at some $x$ then it is non-zero in some interval $J$ containing $x$. Moreover, since everything is also continuous in variables $s$, this interval also serves $s$ which are close enough to $s_1, \ldots , s_{m_{k-l+1}-1}$. Since the set of allowed $s$ is a product of closed intervals, which is compact, we can find a finite subcover and get a finite family of intervals $J_1, J_2, \ldots , J_K\subset I$ such that for all allowed values of $s$ for at least one $1\le j \le K$ the value of $Q_{k-l+1}((g(x+s_1),\ldots , g(x+s_{m_{k-l+1}-1}))$ is non-zero for all $x\in J_j$.

Let us pick a small enough number $\eta > 0$ and subdivide $[u_{m_{k-l+1}},v_{m_{k-l+1}}]$ into intervals of length less than $\frac{\eta}{3}$ and cover $s_{m_{k-l+1}}$ from each of these intervals separately, and then do a finite union. We will denote $y = x+s_{m_{k-l+1}}$, then the quantity we are after is 
\begin{multline*}g(y)Q_{k-l+1}(g(y+s_1-s_{m_{k-l+1}}), g(y+s_2-s_{m_{k-l+1}}),\ldots , g(y+s_{m_{k-l+1}-1}-s_{m_{k-l+1}})+\\
R_{k-l+1}(g(y+s_1-s_{m_{k-l+1}}), g(y+s_2-s_{m_{k-l+1}}),\ldots , g(y+s_{m_{k-l+1}-1}-s_{m_{k-l+1}}).
\end{multline*}
If this is equal to zero then $g(y) = - \frac{R}{Q}$, where $R$ and $Q$ are the expressions from above. Note that we can assume that $Q$ is non-zero for the $j$ that we consider. Also, if $\eta$ is small enough then this holds for $y$ in some non-trivial interval (it is important that this interval is the same regardless of $s_{m_{k-l+1}}$ if we are within the same small subinterval of length less than $\eta$).

Let us pick $m_{k-l+1}+1$ different $y$'s from this interval $y_1 < y_2 < \ldots < y_{m_{k-l+1}+1}$. Note that all of them are also bigger than the time $d+v_{m_{k-l+1}-1}$ up to which we conditioned. We will consider the equation $g(y) = -\frac{R}{Q}$ only for those $y$'s. Note that the right-hands side for them is deterministic, as the values appearing in $R$ and $Q$ depend only on $B(t)$ up to the time $d+v_{m_{k-l+1}-1}$. 

For fixed $y_t$ we will view $-\frac{R}{Q}$ as the function of $s_1, \ldots , s_{m_{k-l+1}}$. This is clearly a smooth function  of these parameters. We will combine these functions for all $t = 1, 2,\ldots , m_{k-l+1}+1$ into a single function from a subset of $\R^{m_{k-l+1}}$ to $\Cm^{m_{k-l+1}+1}$. Since this is a deterministic smooth function and the dimension of the domain $\R^{m_{k-l+1}}$ is smaller than the dimension of the codomain $\Cm^{m_{k-l+1}+1}$, its image has measure $0$. If we can show that the joint distribution of $g(y_1), \ldots , g(y_{m_{k-l+1}+1})$ is absolutely continuous with respect to the Lebesgue measure on $\Cm^{m_{k-l+1}+1}$ it would imply that almost surely for at least one $y_t$ the equality $g(y_t) = -\frac{R}{Q}$ will fail. 

To prove that the joint distribution is absolutely continuous, we advance the time from one $y_t$ to the next. The key idea is that $B(t+s)$ for a given $s>0$ is a Brownian motion with variance $1$ starting at $B(s)$. Since $y_1 > d+v_{m_{k-l+1}+1}$, to get $g(y_1)$ we have to integrate from $0$ to $d+v_{m_{k-l+1}+1}$ and from $d+v_{m_{k-l+1}+1}$ to $y_1$. The first integral is some deterministic constant since $d+v_{m_{k-l+1}+1}$ is the time up to which we conditioned, and the second part is a complex Gaussian by Lemma \ref{brown lemma} (we apply it to the real and imaginary parts separately since they are independent). Hence, the distribution of $g(y_1)$ is absolutely continuous with respect to the Lebesgue measure on $\Cm$.

Next, we are going to condition $B$ up to the time $y_1$ and do the same argument with $y_2$. Regardless of the value of $B(y_1)$ (which plays the role of $r$ in Lemma \ref{brown lemma}) the distribution of $g(y_2)$ is absolutely continuous. So, the distribution of $g(y_2)$ is an average of absolutely continuous distributions, hence it is absolutely continuous for any value of $g(y_1)$. Therefore, the joint distribution of $(g(y_1), g(y_2))$ is absolutely continuous. Continuing in this manner, we will get that the distribution of $(g(y_1), g(y_2), \ldots , g(y_{m_{k-l+1}+1}))$ is absolutely continuous as well, as required.
\section{Rational densities}

For the case of rational $\alpha\beta$ our argument breaks down in two places. First, and less drastically, the arguments of $g$ in the Ron--Shen matrices can coincide: if $\alpha \beta = \frac{m}{n}$ then $x = x - \alpha n + \frac{1}{\beta}m$, so the reasoning with independence and different essential singularities does not go through. However, this is mostly a technical if annoying problem.

The much bigger issue is with the fact that we lose access to the density of the irrational motion. So, instead of demanding that $\det(M_x)$ is non-vanishing somewhere (for the case of continuous functions, say) we need to require that it does not vanish on the whole interval of length $\frac{\alpha}{q}$ if $\alpha\beta = \frac{p}{q}$ with coprime $p, q$. This destroys all our hopes of simply using the general position argument since instead of having infinitely many equations on the function $g$, we now have to make sure that it does not satisfy any equation with an extra added parameter. 

If we nevertheless try to apply our scheme we would need to estimate the number of zeroes of $\det(M_x)$, which can be done but it leads to quite poor bounds on the required denominators of $\alpha\beta$. We present the following theorem without proof as an example of what can be achieved by our methods.
\begin{theorem}
Let $g$ be as in Theorem \ref{main}. For any compact set $K\subset \{(\alpha, \beta)\mid \alpha, \beta > 0, \alpha\beta < 1, \alpha < b - a\}$ there exists a finite set of rational numbers $q_1, q_2, \ldots , q_N$ (depending only on $a, b, K$ but not on $g$) such that for any  $\delta > 0$ there exists $M = M(a, b, K, \delta, g)$ such that if $(\alpha, \beta)\in K$, $|\alpha\beta - q_l|\ge \delta$ for all $l=1, 2, \ldots , N$ and $\alpha\beta = \frac{p}{q}$ for coprime $p, q$ with $q > M$ then $\mathcal{G}(g; \alpha, \beta)$ is a frame.
\end{theorem}

The numbers $q_l$ appear as the ratios $\frac{m}{n}$ where $n, m$ are at most the size of the matrix $M_x$, which is uniformly bounded as long as we are in the compact set $K$ (note that on it $\alpha\beta < 1-\eps$ for some $\eps=\eps(K)>0$). If we are away from these rational numbers, we do not have any collisions and so we can get an upper bound on the number of zeros of $\det(M_x)$ on any given interval $J$. If there are at most $Z$ zeroes then there is an interval of length at least $\frac{|J|}{Z+1}$ which has no zeroes so $M > \frac{Z+1}{\alpha |J|}$ works (note that $\alpha$ is also uniformly bounded from below on a compact set).
\begin{remark}
If we consider for example only pairs $(\alpha, 1)$ then the set of $\alpha$'s such that $\mathcal{G}(g;\alpha, 1)$ is not a frame will be a subset of $\Q$ the only limiting points of which are these numbers $q_l$, and the only limiting point of these numbers $q_l$ are the points $0$ and $1$. 

This is better than what general results which give that the set of $\alpha$'s giving a frame is open \cite{FK}, but it is worse than what we know for rational functions $g$ for which we know that the set of $\alpha$'s itself is an at most countable set of rational numbers which can only converge to $1$, see \cite[Remark 1.5]{BKL}.
\end{remark}
\section*{Acknowledgments} Yurii Belov was supported by the RSF grant 24-11-00087. Aleksei Kulikov was supported by BSF Grant 2020019, ISF Grant 1288/21, and by The Raymond and Beverly Sackler Post-Doctoral Scholarship, and by the VILLUM Centre of Excellence for the Mathematics of Quantum Theory (QMATH) with Grant No.10059. 

%The authors would like to thank Ilya Zlotnikov for helpful comments on the manuscript.

\end{document}